\documentclass[12pt]{article}
\usepackage{latexsym, amssymb, amsthm}
\usepackage{dsfont}

\DeclareMathAlphabet\gothic{U}{euf}{m}{n}

\makeatletter
\def\eqnarray{\stepcounter{equation}\let\@currentlabel=\theequation
\global\@eqnswtrue
\tabskip\@centering\let\\=\@eqncr
$$\halign to \displaywidth\bgroup\hfil\global\@eqcnt\z@
  $\displaystyle\tabskip\z@{##}$&\global\@eqcnt\@ne
  \hfil$\displaystyle{{}##{}}$\hfil
  &\global\@eqcnt\tw@ $\displaystyle{##}$\hfil
  \tabskip\@centering&\llap{##}\tabskip\z@\cr}

\def\endeqnarray{\@@eqncr\egroup
      \global\advance\c@equation\m@ne$$\global\@ignoretrue}

\def\@yeqncr{\@ifnextchar [{\@xeqncr}{\@xeqncr[5pt]}}
\makeatother

\begin{document}
\bibliographystyle{tom}

\newtheorem{lemma}{Lemma}[section]
\newtheorem{thm}[lemma]{Theorem}
\newtheorem{cor}[lemma]{Corollary}
\newtheorem{prop}[lemma]{Proposition}

\theoremstyle{definition}

\newtheorem{remark}[lemma]{Remark}
\newtheorem{exam}[lemma]{Example}
\newtheorem{definition}[lemma]{Definition}

\newcommand{\gota}{\gothic{a}}
\newcommand{\gotb}{\gothic{b}}
\newcommand{\gotc}{\gothic{c}}
\newcommand{\gote}{\gothic{e}}
\newcommand{\gotf}{\gothic{f}}
\newcommand{\gotg}{\gothic{g}}
\newcommand{\gothh}{\gothic{h}}
\newcommand{\gotk}{\gothic{k}}
\newcommand{\gotm}{\gothic{m}}
\newcommand{\gotn}{\gothic{n}}
\newcommand{\gotp}{\gothic{p}}
\newcommand{\gotq}{\gothic{q}}
\newcommand{\gotr}{\gothic{r}}
\newcommand{\gots}{\gothic{s}}
\newcommand{\gott}{\gothic{t}}
\newcommand{\gotu}{\gothic{u}}
\newcommand{\gotv}{\gothic{v}}
\newcommand{\gotw}{\gothic{w}}
\newcommand{\gotz}{\gothic{z}}
\newcommand{\gotA}{\gothic{A}}
\newcommand{\gotB}{\gothic{B}}
\newcommand{\gotG}{\gothic{G}}
\newcommand{\gotL}{\gothic{L}}
\newcommand{\gotS}{\gothic{S}}
\newcommand{\gotT}{\gothic{T}}

\newcounter{teller}
\renewcommand{\theteller}{(\alph{teller})}
\newenvironment{tabel}{\begin{list}%
{\rm  (\alph{teller})\hfill}{\usecounter{teller} \leftmargin=1.1cm
\labelwidth=1.1cm \labelsep=0cm \parsep=0cm}
                      }{\end{list}}

\newcounter{tellerr}
\renewcommand{\thetellerr}{(\roman{tellerr})}
\newenvironment{tabeleq}{\begin{list}%
{\rm  (\roman{tellerr})\hfill}{\usecounter{tellerr} \leftmargin=1.1cm
\labelwidth=1.1cm \labelsep=0cm \parsep=0cm}
                         }{\end{list}}

\newcounter{tellerrr}
\renewcommand{\thetellerrr}{(\Roman{tellerrr})}
\newenvironment{tabelR}{\begin{list}%
{\rm  (\Roman{tellerrr})\hfill}{\usecounter{tellerrr} \leftmargin=1.1cm
\labelwidth=1.1cm \labelsep=0cm \parsep=0cm}
                         }{\end{list}}

\newcounter{proofstep}
\newcommand{\nextstep}{\refstepcounter{proofstep}\vertspace \par 
          \noindent{\bf Step \theproofstep} \hspace{5pt}}
\newcommand{\firststep}{\setcounter{proofstep}{0}\nextstep}

\newcommand{\Ni}{\mathds{N}}
\newcommand{\Qi}{\mathds{Q}}
\newcommand{\Ri}{\mathds{R}}
\newcommand{\Ci}{\mathds{C}}
\newcommand{\Ti}{\mathds{T}}
\newcommand{\Zi}{\mathds{Z}}
\newcommand{\Fi}{\mathds{F}}

\renewcommand{\proofname}{{\bf Proof}}

\newcommand{\vertspace}{\vskip10.0pt plus 4.0pt minus 6.0pt}

\newcommand{\simh}{{\stackrel{{\rm cap}}{\sim}}}
\newcommand{\ad}{{\mathop{\rm ad}}}
\newcommand{\Ad}{{\mathop{\rm Ad}}}
\newcommand{\alg}{{\mathop{\rm alg}}}
\newcommand{\clalg}{{\mathop{\overline{\rm alg}}}}
\newcommand{\Aut}{\mathop{\rm Aut}}
\newcommand{\arccot}{\mathop{\rm arccot}}
\newcommand{\capp}{{\mathop{\rm cap}}}
\newcommand{\rcapp}{{\mathop{\rm rcap}}}
\newcommand{\diam}{\mathop{\rm diam}}
\newcommand{\divv}{\mathop{\rm div}}
\newcommand{\dom}{\mathop{\rm dom}}
\newcommand{\codim}{\mathop{\rm codim}}
\newcommand{\RRe}{\mathop{\rm Re}}
\newcommand{\IIm}{\mathop{\rm Im}}
\newcommand{\tr}{{\mathop{\rm Tr \,}}}
\newcommand{\Tr}{{\mathop{\rm Tr \,}}}
\newcommand{\Vol}{{\mathop{\rm Vol}}}
\newcommand{\card}{{\mathop{\rm card}}}
\newcommand{\rank}{\mathop{\rm rank}}
\newcommand{\supp}{\mathop{\rm supp}}
\newcommand{\sgn}{\mathop{\rm sgn}}
\newcommand{\essinf}{\mathop{\rm ess\,inf}}
\newcommand{\esssup}{\mathop{\rm ess\,sup}}
\newcommand{\Int}{\mathop{\rm Int}}
\newcommand{\lcm}{\mathop{\rm lcm}}
\newcommand{\loc}{{\rm loc}}
\newcommand{\HS}{{\rm HS}}
\newcommand{\Trn}{{\rm Tr}}
\newcommand{\n}{{\rm N}}
\newcommand{\WOT}{{\rm WOT}}

\newcommand{\at}{@}

\newcommand{\mod}{\mathop{\rm mod}}
\newcommand{\spann}{\mathop{\rm span}}
\newcommand{\one}{\mathds{1}}

\hyphenation{groups}
\hyphenation{unitary}

\newcommand{\tfrac}[2]{{\textstyle \frac{#1}{#2}}}

\newcommand{\ca}{{\cal A}}
\newcommand{\cb}{{\cal B}}
\newcommand{\cc}{{\cal C}}
\newcommand{\cd}{{\cal D}}
\newcommand{\ce}{{\cal E}}
\newcommand{\cf}{{\cal F}}
\newcommand{\ch}{{\cal H}}
\newcommand{\chs}{{\cal HS}}
\newcommand{\ci}{{\cal I}}
\newcommand{\ck}{{\cal K}}
\newcommand{\cl}{{\cal L}}
\newcommand{\cm}{{\cal M}}
\newcommand{\cn}{{\cal N}}
\newcommand{\co}{{\cal O}}
\newcommand{\cp}{{\cal P}}
\newcommand{\cs}{{\cal S}}
\newcommand{\ct}{{\cal T}}
\newcommand{\cx}{{\cal X}}
\newcommand{\cy}{{\cal Y}}
\newcommand{\cz}{{\cal Z}}

\newlength{\hightcharacter}
\newlength{\widthcharacter}
\newcommand{\covsup}[1]{\settowidth{\widthcharacter}{$#1$}\addtolength{\widthcharacter}{-0.15em}\settoheight{\hightcharacter}{$#1$}\addtolength{\hightcharacter}{0.1ex}#1\raisebox{\hightcharacter}[0pt][0pt]{\makebox[0pt]{\hspace{-\widthcharacter}$\scriptstyle\circ$}}}
\newcommand{\cov}[1]{\settowidth{\widthcharacter}{$#1$}\addtolength{\widthcharacter}{-0.15em}\settoheight{\hightcharacter}{$#1$}\addtolength{\hightcharacter}{0.1ex}#1\raisebox{\hightcharacter}{\makebox[0pt]{\hspace{-\widthcharacter}$\scriptstyle\circ$}}}
\newcommand{\scov}[1]{\settowidth{\widthcharacter}{$#1$}\addtolength{\widthcharacter}{-0.15em}\settoheight{\hightcharacter}{$#1$}\addtolength{\hightcharacter}{0.1ex}#1\raisebox{0.7\hightcharacter}{\makebox[0pt]{\hspace{-\widthcharacter}$\scriptstyle\circ$}}}

\thispagestyle{empty}

\vspace*{1cm}
\begin{center}
{\Large\bf The Dirichlet problem  \\[1mm]
 without the maximum principle} \\[10mm]

\large W. Arendt$^1$ and A.F.M. ter Elst$^2$

\end{center}

\vspace{5mm}

\begin{center}
{\bf Abstract}
\end{center}

\begin{list}{}{\leftmargin=1.7cm \rightmargin=1.7cm \listparindent=10mm 
   \parsep=0pt}
\item
Consider the Dirichlet problem with respect to an elliptic operator
\[
A = - \sum_{k,l=1}^d \partial_k \, a_{kl} \, \partial_l
   - \sum_{k=1}^d \partial_k \, b_k
   + \sum_{k=1}^d c_k \, \partial_k
   + c_0
\]
on a bounded Wiener regular open set $\Omega \subset \Ri^d$,
where $a_{kl}, c_k \in L_\infty(\Omega,\Ri)$ and 
$b_k,c_0 \in L_\infty(\Omega,\Ci)$.
Suppose that the associated operator on $L_2(\Omega)$ with 
Dirichlet boundary conditions is invertible.
Then we show that for all $\varphi \in C(\partial \Omega)$ there exists a 
unique $u \in C(\overline \Omega) \cap H^1_\loc(\Omega)$ such that 
$u|_{\partial \Omega} = \varphi$ and $A u = 0$.

In the case when $\Omega$ has a Lipschitz boundary and 
$\varphi \in C(\overline \Omega) \cap H^{1/2}(\overline \Omega)$, then we 
show that $u$ coincides with the variational solution in $H^1(\Omega)$.
\end{list}

\vspace{1cm}
\noindent
October 2017

\vspace{5mm}

\vspace{5mm}

\vspace{15mm}

\noindent
{\bf Home institutions:}    \\[3mm]
\begin{tabular}{@{}cl@{\hspace{10mm}}cl}
1. & Institute of Applied Analysis  & 
  2. & Department of Mathematics   \\
& University of Ulm   & 
  & University of Auckland   \\
& Helmholtzstr.\ 18 & 
  & Private bag 92019  \\
& 89081 Ulm & 
  & Auckland  \\
& Germany  & 
  & New Zealand  \\[8mm]
\end{tabular}

\newpage

\section{Introduction} \label{Dwreg1}

Let $\Omega \subset \Ri^d$ be an open bounded set with boundary~$\Gamma$.
Throughout this paper we assume that $d \geq 2$.
The classical Dirichlet problem is to find for each $\varphi \in C(\Gamma)$
a function $u \in C(\overline \Omega)$ 
such that $u|_\Gamma = \varphi$ and $\Delta u = 0$ as distribution on $\Omega$.
The set $\Omega$ is called {\bf Wiener regular} if for every 
$\varphi \in C(\Gamma)$ there exists a unique $u \in C(\overline \Omega)$ 
such that $u|_\Gamma = \varphi$ and $\Delta u = 0$ as distribution on $\Omega$.

The Dirichlet problem has been extended naturally to more general
second-order operators.
For all $k,l \in \{ 1,\ldots,d \} $ let $a_{kl} \colon \Omega \to \Ri$ 
be a bounded measurable function and suppose that there exists a
$\mu > 0$ such that 
\begin{equation}
\RRe \sum_{k,l=1}^d a_{kl}(x) \, \xi_k \, \overline{\xi_l} 
\geq \mu \, |\xi|^2
\label{eSwreg1;2}
\end{equation}
for all $x \in \Omega$ and $\xi \in \Ci^d$.
Further, for all $k \in \{ 1,\ldots,d \} $ let $b_k,c_k,c_0 \colon \Omega \to \Ci$
be bounded and measurable.
Define the map $\ca \colon H^1_\loc(\Omega) \to \cd'(\Omega)$ by
\[
\langle \ca u,v \rangle_{\cd'(\Omega) \times \cd(\Omega)}
= \sum_{k,l=1}^d \int_\Omega a_{kl} \, (\partial_k u) \, \overline{\partial_l v}
   + \sum_{k=1}^d \int_\Omega b_k \, u \, \overline{\partial_k v}
   + \sum_{k=1}^d \int_\Omega c_k \, (\partial_k u) \, \overline v
   + \int_\Omega c_0 \, u \, \overline v
\]
for all $u \in H^1_\loc(\Omega)$ and $v \in C_c^\infty(\Omega)$.
Given $\varphi \in C(\Gamma)$, by a {\bf classical solution} of the 
Dirichlet problem we understand a function 
$u \in C(\overline \Omega) \cap H^1_\loc(\Omega)$ satisfying $\ca u = 0$
and $u|_\Gamma = \varphi$.
For the pure second-order case (that is $b_k = c_k = c_0 = 0$)
Littman--Stampacchia--Weinberger
\cite{LSW} proved that for all $\varphi \in C(\Gamma)$
there exists a unique classical solution~$u$.
Then Stampacchia \cite{Stam2} Th\'eor\`eme~10.2
added real valued lower order terms, under the condition 
(see \cite{Stam2}, (9.2')) that there exists a 
$\mu' > 0$
such that 
\begin{equation}
\int_\Omega c_0 \, v + \sum_{k=1}^d \int_\Omega b_k \, \partial_k v
\geq \mu' \int_\Omega v
\label{eSwreg1;3.1}
\end{equation}
for all $v \in C_c^\infty(\Omega)^+$.
Gilbarg--Trudinger \cite{GT} Theorem~8.31 merely assume that 
\begin{equation}
\int_\Omega c_0 \, v + \sum_{k=1}^d \int_\Omega b_k \, \partial_k v
\geq 0
\label{eSwreg1;3}
\end{equation}
for all $v \in C_c^\infty(\Omega)^+$ in order to obtain the same conclusion.
A consequence of these assumptions is a weak maximum principle,
which implies that $\|u\|_{C(\overline \Omega)} \leq \|\varphi\|_{C(\Gamma)}$
for all $u \in H^1_\loc(\Omega) \cap C(\overline \Omega)$ satisfying $\ca u = 0$
and $u|_\Gamma = \varphi$.
We may consider (\ref{eSwreg1;3}) as a kind of submarkov condition since 
it is equivalent to $-\ca \one_\Omega \leq 0$ in $\cd'(\Omega)$.

The aim of this paper is to 
show that the positivity condition (\ref{eSwreg1;3}) and the maximum principle
are not needed for the well-posedness of the Dirichlet problem.
In addition we allow the $b_k$ and $c_0$ to be complex valued.
In order to state the main results of this paper in a more 
precise way we need a few definitions.
Define the form $\gota \colon H^1(\Omega) \times H^1(\Omega) \to \Ci$ by
\begin{equation}
\gota(u,v) 
= \sum_{k,l=1}^d \int_\Omega a_{kl} \, (\partial_k u) \, \overline{\partial_l v}
   + \sum_{k=1}^d \int_\Omega b_k \, u \, \overline{\partial_k v}
   + \sum_{k=1}^d \int_\Omega c_k \, (\partial_k u) \, \overline v
   + \int_\Omega c_0 \, u \, \overline v
.  
\label{eSwreg1;4}
\end{equation}
Let $A^D$ be the operator in $L_2(\Omega)$ associated with the 
form $\gota|_{H^1_0(\Omega) \times H^1_0(\Omega)}$.
In other words, $A^D$ is the realisation of the elliptic operator $\ca$ 
in $L_2(\Omega)$ with Dirichlet boundary conditions.
This operator has a compact resolvent.
Moreover, if (\ref{eSwreg1;3}) is valid, then $\ker A^D = \{ 0 \} $ by \cite{GT} 
Corollary~8.2.
Instead of (\ref{eSwreg1;3}) we assume the condition $\ker A^D = \{ 0 \} $,
which is equivalent to the uniqueness of the Dirichlet problem
(cf.\ Proposition~\ref{pwreg202.5} below).

The main result of this paper is the following well-posedness result 
for the Dirichlet problem.

\begin{thm} \label{twreg101}
Let $\Omega \subset \Ri^d$ be an open bounded Wiener regular set with $d \geq 2$.
For all $k,l \in \{ 1,\ldots,d \} $ let $a_{kl} \colon \Omega \to \Ri$ 
be a bounded measurable function and suppose that there exists a
$\mu > 0$ such that 
\[
\RRe \sum_{k,l=1}^d a_{kl}(x) \, \xi_k \, \overline{\xi_l} 
\geq \mu \, |\xi|^2
\]
for all $x \in \Omega$ and $\xi \in \Ci^d$.
Further, for all $k \in \{ 1,\ldots,d \} $ let $b_k,c_0 \colon \Omega \to \Ci$
and $c_k \colon \Omega \to \Ri$
be bounded and measurable.
Let $A^D$ be as above.
Suppose $0 \not\in \sigma(A^D)$.
Then for all $\varphi \in C(\Gamma)$
there exists a unique $u \in C(\overline \Omega) \cap H^1_\loc(\Omega)$
such that $u|_\Gamma = \varphi$ and $\ca u = 0$.

Moreover, there exists a constant $c > 0$ such that 
\[
\|u\|_{C(\overline \Omega)} \leq c \, \|\varphi\|_{C(\Gamma)}
\]
for all $\varphi \in C(\Gamma)$,
where $u \in C(\overline \Omega) \cap H^1_\loc(\Omega)$
is such that $u|_\Gamma = \varphi$ and $\ca u = 0$.
\end{thm}

Instead of the homogeneous equation $\ca u = 0$ one can also consider
the inhomogeneous equation $\ca u = f_0 + \sum_{k=1}^d \partial_k f_k$.
We shall do that in Theorem~\ref{twreg216}.

\smallskip

Adopt the notation and assumptions of Theorem~\ref{twreg101}.
Define $P \colon C(\Gamma) \to C(\overline \Omega)$ by $P \varphi = u$,
where $u \in C(\overline \Omega) \cap H^1_\loc(\Omega)$
is such that $u|_\Gamma = \varphi$ and $\ca u = 0$.
Note that $P \varphi$ is the {\bf classical solution} of the Dirichlet problem.

If $\Omega$ has even a Lipschitz boundary
(which implies Wiener regularity), then there is also a variational
solution of the Dirichlet problem that we describe next.
Denote by $\Tr \colon H^1(\Omega) \to L_2(\Gamma)$ the trace operator.
Again let $a_{kl},b_k,c_k,c_0 \in L_\infty(\Omega)$ and 
suppose that the ellipticity condition (\ref{eSwreg1;2}) is satisfied.
Further suppose that $0 \not\in \sigma(A^D)$.
Then for each $\varphi \in \Tr H^1(\Omega)$
there exists a unique $u \in H^1(\Omega)$, called the 
{\bf variational solution}, such that 
$\ca u = 0$ and $\Tr u = \varphi$ (cf.\ Lemma~\ref{lwreg202}).
Define $\gamma \colon \Tr H^1(\Omega) \to H^1(\Omega)$ by setting 
$\gamma \varphi = u$.

The second result of this paper says that the variational
solution and the classical solution coincide, if both are defined.

\begin{thm} \label{twreg102}
Adopt the notation and assumptions of Theorem~\ref{twreg101}.
Suppose that $\Omega$ has a Lipschitz boundary.
Let $\varphi \in C(\Gamma) \cap \Tr H^1(\Omega)$.
Then $P \varphi = \gamma \varphi$ almost everywhere on $\Omega$.
\end{thm}

The last main result of this paper concerns a parabolic equation.
Let $A_c$ denote the part of the operator $A^D$ in $C_0(\Omega)$.
So 
\[
D(A_c) = \{ u \in D(A^D) \cap C_0(\Omega) : A^D u \in C_0(\Omega) \}
\]
and $A_c = A^D|_{D(A_c)}$.

\begin{thm} \label{twreg103}
Adopt the notation and assumptions of Theorem~\ref{twreg101}.
Then $-A_c$ generates a holomorphic $C_0$-semigroup on $C_0(\Omega)$.
Moreover, $e^{-t A_c} u = e^{-t A^D} u$ for all $u \in C_0(\Omega)$
and $t > 0$.
\end{thm}

In Section~\ref{Swreg2} we prove Theorem~\ref{twreg101} via 
an iteration argument.
Section~\ref{Swreg3new} is devoted to the comparison of the classical
and the variational solutions of the Dirichlet problem.
Theorem~\ref{twreg102} is proved there with the help of a deep 
result of Dahlberg \cite{Dahlberg}.
We consider the semigroup on $C_0(\Omega)$ in Section~\ref{Swreg4new} and 
prove Theorem~\ref{twreg103}.

\section{The Dirichlet problem} \label{Swreg2}

In this section we prove Theorem~\ref{twreg101} on the well-posedness
of the Dirichlet problem.
The technique is a reduction to the Stampacchia result mentioned in the 
introduction.
For this reason we introduce the following two forms and operators.

Adopt the notation and assumptions of Theorem~\ref{twreg101}.
For all $\lambda \in \Ri$ define the forms 
$\gota_\lambda, \gotb_\lambda \colon H^1(\Omega) \times H^1(\Omega) \to \Ci$
by 
\begin{eqnarray*}
\gota_\lambda(u,v) 
& = & \gota(u,v) + \lambda \, (u,v)_{L_2(\Omega)}  \quad \mbox{and}  \\
\gotb_\lambda(u,v) 
& = & \sum_{k,l=1}^d \int_\Omega a_{kl} \, (\partial_k u) \, \overline{\partial_l v}
   + \sum_{k=1}^d \int_\Omega c_k \, (\partial_k u) \, \overline v
   + \lambda \int_\Omega u \, \overline v
,
\end{eqnarray*}
where $\gota$ is as in (\ref{eSwreg1;4}).
Define similarly $\ca_\lambda,\cb_\lambda \colon H^1_\loc(\Omega) \to \cd'(\Omega)$
and let $B^D$ be the operator associated with the sesquilinear form
$\gotb_0|_{H^1_0(\Omega) \times H^1_0(\Omega)}$.
It follows from ellipticity that there exists a $\lambda_0 > 0$ 
such that 
\[
\frac{\mu}{2} \, \|v\|_{H^1(\Omega)}^2
\leq \RRe \gota_{\lambda_0}(v)
\quad \mbox{and} \quad
\frac{\mu}{2} \, \|v\|_{H^1(\Omega)}^2
\leq \RRe \gotb_{\lambda_0}(v)
\]
for all $v \in H^1(\Omega)$.
Note that $\cb_\lambda$ satisfies the submarkovian condition 
$- \cb_\lambda \one_\Omega \leq 0$, that is (\ref{eSwreg1;3}),
and even Stampacchia's condition (\ref{eSwreg1;3.1})
for all $\lambda > 0$.
So we can and will apply Stampacchia's result (in the proof of 
Lemma~\ref{lwreg210}).

We first investigate the operator $A^D$ in $L_2(\Omega)$.
Note that $f_0 + \sum_{k=1}^d \partial_k f_k \in \cd'(\Omega)$
for all $f_0,f_1,\ldots,f_d \in L_1(\Omega)$.
The next lemma is also valid if the $a_{kl}$ and $c_k$ are complex valued.

\begin{lemma} \label{lwreg202}
Let $f_1,\ldots,f_d \in L_2(\Omega)$.
Let $\tilde p \in (1,\infty)$ be such that 
$\tilde p \geq \frac{2d}{d+2}$.
Further let $f_0 \in L_{\tilde p}(\Omega)$.
Then there exists a unique $u \in H^1_0(\Omega)$ such that 
$\ca u = f_0 + \sum_{k=1}^d \partial_k f_k$.
\end{lemma}
\begin{proof}
There exists a unique $T \in \cl(H^1_0(\Omega))$ such that 
$(T u,v)_{H^1_0(\Omega)} = \gota(u,v)$
for all $u,v \in H^1_0(\Omega)$.
Then $T$ is injective because $\ker A^D = \{ 0 \} $.
Moreover, the inclusion $H^1_0(\Omega) \hookrightarrow L_2(\Omega)$ is compact.
Hence the operator $T$ is invertible by the Fredholm--Lax--Milgram lemma,
\cite{AEKS} Lemma~4.1.
Clearly $v \mapsto \sum_{k=1}^d (f_k, \partial_k v)_{L_2(\Omega)}$ 
is continuous from $H^1_0(\Omega)$ into $\Ci$.
Define $F \colon C_c^\infty(\Omega) \to \Ci$ by 
$F(v) = \langle f_0,v \rangle_{\cd'(\Omega) \times \cd(\Omega)}$.
We claim that $F$
extends to a continuous function from $H^1_0(\Omega)$ into $\Ci$.
If $d \geq 3$, then $H^1_0(\Omega) \subset L_r(\Omega)$, where
$r = \frac{2d}{d-2}$.
So $H^1_0(\Omega) \subset L_q(\Omega)$, where $q$ is the dual 
exponent of $\tilde p$.
The last inclusion is also valid if $d = 2$.
So in any case the map $F$ extends to a continuous function from 
$H^1_0(\Omega)$ into $\Ci$.
Then the lemma follows.
\end{proof}

The next lemma is valid for a general bounded open set $\Omega$
and does not use the condition $0 \not\in \sigma(A^D)$.
It is an extension of \cite{ABenilan} Lemma~4.2.

\begin{lemma} \label{lwreg205}
Let $u \in C_0(\Omega) \cap H^1_\loc(\Omega)$ and 
$f_1,\ldots,f_d \in L_2(\Omega)$.
Let $\tilde p \in (1,\infty)$ be such that 
$\tilde p \geq \frac{2d}{d+2}$.
Further let $f_0 \in L_{\tilde p}(\Omega)$.
Suppose that $\ca u = f_0 + \sum_{k=1}^d \partial_k f_k$.
Then $u \in H^1_0(\Omega)$.
\end{lemma}
\begin{proof}
As at the end of the previous proof there exists an $M_0 > 0$ such that 
$|\int_\Omega f_0 \, \overline v| \leq M_0 \, \|v\|_{H^1(\Omega)}$ for all 
$v \in H^1_0(\Omega)$.
Set $M = M_0 + \sum_{k=1}^d \|f_k\|_2$.

Let $\varepsilon > 0$.
Set $v_\varepsilon = (\RRe u-\varepsilon)^+$.
Then $\supp v_\varepsilon \subset \Omega$ is compact.
Hence there exists an open $\Omega_1 \subset \Ri^d$ such that 
$\supp v_\varepsilon 
\subset \Omega_1 \subset \overline{\Omega_1} \subset \Omega$.
Then $v_\varepsilon \in H^1_0(\Omega_1)$.
Moreover,
\begin{eqnarray}
\lefteqn{
\sum_{k,l=1}^d \int_{\Omega_1} a_{kl} \, (\partial_k u) \, \overline{\partial_l v}
   + \sum_{k=1}^d \int_{\Omega_1} b_k \, u \, \overline{\partial_k v}
   + \sum_{k=1}^d \int_{\Omega_1} c_k \, (\partial_k u) \, \overline v
   + \int_{\Omega_1} c_0 \, u \, \overline v
} \hspace*{90mm}  \nonumber  \\*
& = & \int_{\Omega_1} f_0  \, \overline v
   + \sum_{k=1}^d \int_{\Omega_1} f_k \, \overline{\partial_k v}
\label{elwreg205;1}
\end{eqnarray}
for all $v \in C_c^\infty(\Omega_1)$.
Since $u|_{\Omega_1} \in H^1(\Omega_1)$ it follows that (\ref{elwreg205;1})
is valid for all $v \in H^1_0(\Omega_1)$.
Choosing $v = v_\varepsilon$ gives
\begin{eqnarray*}
\lefteqn{
\Big| \sum_{k,l=1}^d \int_\Omega a_{kl} \, (\partial_k u) \, \partial_l v_\varepsilon
   + \sum_{k=1}^d \int_\Omega b_k \, u \, \partial_k v_\varepsilon
   + \sum_{k=1}^d \int_\Omega c_k \, (\partial_k u) \, v_\varepsilon
   + \int_\Omega c_0 \, u \, v_\varepsilon   \Big|  
} \hspace*{60mm}  \\*
& \leq & M_0 \, \|v_\varepsilon\|_{H^1(\Omega)}
   + \sum_{k=1}^d \|f_k\|_2 \, \|\partial_k v_\varepsilon\|_2
\leq M \, \|v_\varepsilon\|_{H^1(\Omega)}
.
\end{eqnarray*}
On the other hand, 
$\partial_k v_\varepsilon 
= \partial_k ((\RRe u-\varepsilon)^+) 
= \one_{[\RRe u > \varepsilon]} \, \partial_k \RRe u$
for all $k \in \{ 1,\ldots,d \} $ by \cite{GT} Lemma~7.6.
Therefore
\begin{eqnarray*}
\lefteqn{
\RRe \sum_{k,l=1}^d \int_\Omega a_{kl} \, (\partial_k u) \, \partial_l v_\varepsilon
   + \RRe \sum_{k=1}^d \int_\Omega b_k \, u \, \partial_k v_\varepsilon
   + \RRe \sum_{k=1}^d \int_\Omega c_k \, (\partial_k u) \, v_\varepsilon
   + \RRe \int_\Omega c_0 \, u \, v_\varepsilon  
} \hspace*{10mm}  \\*
& = & 
\sum_{k,l=1}^d \int_\Omega a_{kl} \, (\partial_k v_\varepsilon) \, \partial_l v_\varepsilon
   + \RRe \sum_{k=1}^d \int_\Omega b_k \, u \, \partial_k v_\varepsilon
   + \sum_{k=1}^d \int_\Omega c_k \, (\partial_k \RRe u) \, v_\varepsilon
   + \RRe \int_\Omega c_0 \, u \, v_\varepsilon  \\
& = & \RRe \gota(v_\varepsilon)
   + \varepsilon \sum_{k=1}^d \int_\Omega (\RRe b_k) \, \partial_k v_\varepsilon
   - \sum_{k=1}^d \int_\Omega (\IIm b_k) \, (\IIm u) \, \partial_k v_\varepsilon
  \\*
& & {} \hspace*{30mm}
   + \varepsilon \int_\Omega (\RRe c_0) \, v_\varepsilon  
   - \int_\Omega (\IIm c_0) \, (\IIm u) \, v_\varepsilon  \\
& \geq & \frac{\mu}{2} \, \|v_\varepsilon\|_{H^1(\Omega)}^2 
   - \lambda_0 \, \|v_\varepsilon\|_2^2 
   - \varepsilon \, M' \, |\Omega|^{1/2} \, \|v_\varepsilon\|_{H^1(\Omega)}
   - M' \, \|u\|_2 \, \|v_\varepsilon\|_{H^1(\Omega)}
,
\end{eqnarray*}
where $M' = \|c_0\|_\infty + \sum_{k=1}^d \|b_k\|_\infty$.
Since 
$\|v_\varepsilon\|_2 
= \|(\RRe u-\varepsilon)^+\|_2 
\leq \|u\|_2 
\leq |\Omega|^{1/2} \, \|u\|_{C_0(\Omega)}$,
it follows that 
\[
\frac{\mu}{2} \, \|(\RRe u-\varepsilon)^+\|_{H^1(\Omega)}^2
\leq M'' \, \|(\RRe u-\varepsilon)^+\|_{H^1(\Omega)}
   + \lambda_0 \, |\Omega| \, \|u\|_{C_0(\Omega)}^2
\]
for all $\varepsilon \in (0,1]$, where 
$M'' = M + M' \, |\Omega|^{1/2} \, (\|u\|_{C_0(\Omega)} + 1)$.

Therefore the sequence $((\RRe u - 2^{-n})^+)_{n \in \Ni_0}$ is bounded in 
$H^1_0(\Omega)$.
Passing to a subsequence if necessary, we may assume without loss of generality
that there exists a $w \in H^1_0(\Omega)$ such that 
$\lim (\RRe u - 2^{-n})^+ = w$ weakly in $H^1_0(\Omega)$.
Then $\lim (\RRe u - 2^{-n})^+ = w$ in $L_2(\Omega)$.
But $\lim (\RRe u - 2^{-n})^+ = (\RRe u)^+$ in $L_2(\Omega)$.
So $(\RRe u)^+ = w \in H^1_0(\Omega)$.
Similarly $(\RRe u)^-, (\IIm u)^+, (\IIm u)^- \in H^1_0(\Omega)$.
So $u \in H^1_0(\Omega)$.
\end{proof}

Lemma~\ref{lwreg205} together with the condition $0 \not\in \sigma(A^D)$
gives the uniqueness in Theorem~\ref{twreg101}.

\begin{prop} \label{pwreg202.5}
For all $\varphi \in C(\Gamma)$ there exists at most one 
$u \in C(\overline \Omega) \cap H^1_\loc(\Omega)$
such that $u|_\Gamma = \varphi$ and $\ca u = 0$.
\end{prop}
\begin{proof}
Let $u \in C(\overline \Omega) \cap H^1_\loc(\Omega)$
and suppose that $u|_\Gamma = 0$ and $\ca u = 0$.
Then $u \in C_0(\Omega)$.
Hence $u \in H^1_0(\Omega)$ by Lemma~\ref{lwreg205}.
Also $\ca u = 0$.
Therefore $u \in D(A^D)$ and $A^D u = 0$.
But $0 \not\in \sigma(A^D)$.
So $u = 0$.
\end{proof}

In the next proposition we use that $\Omega$ is Wiener regular.

\begin{prop} \label{pwreg203}
Let $\lambda > \lambda_0$ and $p \in (d,\infty]$.
Let $f_0 \in L_{p/2}(\Omega)$ and $f_1,\ldots,f_d \in L_p(\Omega)$.
Then there exists a unique $u \in H^1_0(\Omega) \cap C_0(\Omega)$
such that $\cb_\lambda u = f_0 + \sum_{k=1}^d \partial_k f_k$.
\end{prop}
\begin{proof}
Since $a_{kl}$ and $c_k$ are real valued for all $k,l \in \{ 1,\ldots,d \} $
we may assume that $f_0,\ldots,f_d$ are real valued.
By \cite{GT} Theorem~8.31 there exists a unique $u \in C(\overline \Omega) \cap H^1_\loc(\Omega)$
such that $\cb_\lambda u = f_0 + \sum_{k=1}^d \partial_k f_k$ and $u|_\Gamma = 0$.
Then $u \in C_0(\Omega)$ and the existence follows from Lemma~\ref{lwreg205}.
The uniqueness follows from Proposition~\ref{pwreg202.5}.
\end{proof}

\begin{cor} \label{cwreg206}
Let $\lambda > \lambda_0$ and $p \in (d,\infty]$.
Let $f_0 \in L_{p/2}(\Omega)$ and $f_1,\ldots,f_d \in L_p(\Omega)$.
Let $u \in H^1_0(\Omega)$ and suppose that $\cb_\lambda u = f_0 + \sum_{k=1}^d \partial_k f_k$.
Then $u \in C_0(\Omega)$.
\end{cor}
\begin{proof}
By Proposition~\ref{pwreg203} there exists a $\tilde u \in H^1_0(\Omega) \cap C_0(\Omega)$
such that $\cb_\lambda \tilde u = f_0 + \sum_{k=1}^d \partial_k f_k$.
Then $\cb_\lambda(u - \tilde u) = 0$.
So $\gotb_\lambda(u - \tilde u,v) = 0$ first for all $v \in C_c^\infty(\Omega)$ and 
then by density for all $v \in H^1_0(\Omega)$.
Choose $v = u - \tilde u$.
Then $\frac{\mu}{2} \, \|u - \tilde u\|_{H^1(\Omega)}^2 \leq \RRe \gotb_\lambda(u - \tilde u) = 0$.
So $u = \tilde u \in C_0(\Omega)$.
\end{proof}

We next wish to add the other lower order terms.

\begin{prop} \label{pwreg209}
There exists a $c > 0$ such that 
for all $\Phi \in C^1(\Ri^d)$ there exists a unique
$u \in H^1(\Omega) \cap C(\overline \Omega)$ such that 
$u|_\Gamma = \Phi|_\Gamma$ and $\ca u = 0$.
Moreover,
\[
\|u\|_{C(\overline \Omega)}
\leq c \, \|\Phi|_\Gamma\|_{C(\Gamma)}
.  \]
\end{prop}

For the proof we need some lemmas.
In the next lemma we introduce a parameter $\delta$ in order
to avoid duplication of the proof.

\begin{lemma} \label{lwreg207}
Fix $\delta \in [0,\lambda_0 + 1]$.
\mbox{}
\begin{tabel}
\item \label{lwreg207-1}
For all $f \in L_2(\Omega)$ and $\lambda > \lambda_0$ 
there exists a unique $u \in H^1_0(\Omega)$ such that
\begin{equation}
\gotb_\lambda(u,v) 
= \sum_{k=1}^d (b_k \, f, \partial_k v)_{L_2(\Omega)} 
   + \, ((c_0 - \delta \, \one_\Omega) \, f, v)_{L_2(\Omega)}
\label{elwreg207;1}
\end{equation}
for all $v \in H^1_0(\Omega)$.
\end{tabel}
For all $\lambda > \lambda_0$ define 
$R_\lambda \colon L_2(\Omega) \to L_2(\Omega)$ by $R_\lambda f = u$, where 
$u \in H^1_0(\Omega)$ is as in~{\rm (\ref{elwreg207;1})}.
\begin{tabel}
\setcounter{teller}{1}
\item \label{lwreg207-1.5}
There exists a $c_1 > 0$ such that 
\[
\|R_\lambda f\|_{L_q(\Omega)} \leq c_1 \, (\lambda - \lambda_0)^{-1/4} \, \|f\|_{L_2(\Omega)}
\]
for all $\lambda > \lambda_0$ and $f \in L_2(\Omega)$,
where $\frac{1}{q} = \frac{1}{2} - \frac{1}{4d}$.
\item \label{lwreg207-2}
There exists a $c_2 \geq 1$ such that 
\[
\|R_\lambda f\|_{L_q(\Omega)} \leq c_2 \, \|f\|_{L_p(\Omega)}
\]
for all $\lambda \in [\lambda_0 + 1,\infty)$,
$p,q \in [2,\infty]$ and $f \in L_p(\Omega)$ with $\frac{1}{q} = \frac{1}{p} - \frac{1}{4d}$.
\item \label{lwreg207-3}
If $\lambda > \lambda_0$, 
$p \in (d,\infty]$ and $f \in L_p(\Omega)$, then $R_\lambda f \in C_0(\Omega)$.
\end{tabel}
\end{lemma}
\begin{proof}
`\ref{lwreg207-1}'.
This follows from the Lax--Milgram theorem.

`\ref{lwreg207-1.5}'.
Define $M = \|c_0 - \delta \, \one_\Omega\|_{L_\infty(\Omega)} + \sum_{k=1}^d \|b_k\|_{L_\infty(\Omega)}$.
Let $\lambda > \lambda_0$, $f \in L_2(\Omega)$ and set $u = R_\lambda f$.
Then 
\begin{eqnarray*}
\frac{\mu}{2} \, \|u\|_{H^1(\Omega)}^2 + (\lambda - \lambda_0) \|u\|_{L_2(\Omega)}^2
& \leq & \RRe \gotb_{\lambda_0}(u) + (\lambda - \lambda_0) \|u\|_{L_2(\Omega)}^2  \\
& = & \RRe \gotb_\lambda(u)  \\
& = & \RRe \sum_{k=1}^d (b_k \, f, \partial_k u)_{L_2(\Omega)} 
   + \RRe ((c_0 - \delta \, \one_\Omega) \, f, u)_{L_2(\Omega)}  \\
& \leq & M \, \|f\|_{L_2(\Omega)} \, \|u\|_{H^1(\Omega)}
.  
\end{eqnarray*}
So 
$\|u\|_{H^1(\Omega)} 
\leq 2 \mu^{-1} \, M \, \|f\|_{L_2(\Omega)}$
and 
\[
\|R_\lambda f\|_{L_2(\Omega)}
= \|u\|_{L_2(\Omega)}
\leq \sqrt{\frac{2}{\mu (\lambda - \lambda_0)} } \, M \, \|f\|_{L_2(\Omega)}
.  \]
By the Sobolev embedding theorem there exists a $c_1 > 0$ such that 
$\|v\|_{L_{q_1}(\Omega)} \leq c_1 \, \|v\|_{H^1(\Omega)}$ for all $v \in H^1_0(\Omega)$,
where $\frac{1}{q_1} = \frac{1}{2} - \frac{1}{2d}$.
(The extra factor $2$ is to avoid a separate case for $d=2$.)
Then 
$\|R_\lambda f\|_{L_{q_1}(\Omega)} 
\leq 2 \mu^{-1} \, c_1 \, M \, \|f\|_{L_2(\Omega)}$.
Hence 
\[
\|R_\lambda f\|_{L_q(\Omega)} 
\leq \|R_\lambda f\|_{L_2(\Omega)}^{1/2} \, \|R_\lambda f\|_{L_{q_1}(\Omega)}^{1/2}
\leq c_2 \, (\lambda - \lambda_0)^{-1/4} \, \|f\|_{L_2(\Omega)}
,  \]
where $c_2 = (2/\mu)^{3/4} \, c_1^{1/2} \, M$.

`\ref{lwreg207-2}'.
Apply Corollary~\ref{cwreg206} with $p = 4d$ and $\lambda = \lambda_0 + 1$.
It follows that $R_{\lambda_0 + 1} f \in C_0(\Omega)$ for all $f \in L_p(\Omega)$.
Clearly the map $R_{\lambda_0 + 1} |_{L_p(\Omega)} \colon L_p(\Omega) \to C_0(\Omega)$
has a closed graph.
Hence it is continuous.
In particular, there exists a $c_3 > 0$ such that 
$\|R_{\lambda_0 + 1} f\|_{L_\infty(\Omega)} 
= \|R_{\lambda_0 + 1} f\|_{C_0(\Omega)} 
\leq c_3 \, \|f\|_{L_p(\Omega)}$
for all $f \in L_p(\Omega)$.

Let $\lambda \geq \lambda_0 + 1$ and $f \in L_2(\Omega)$.
Write $u = R_\lambda f$ and $u_0 = R_{\lambda_0 + 1} f$.
Then $\gotb_\lambda(u,v) = \gotb_{\lambda_0 + 1}(u_0,v)$
and therefore 
$\gotb_\lambda(u - u_0, v) = - (\lambda - \lambda_0 - 1) \, (u,v)_{L_2(\Omega)}$
for all $v \in H^1_0(\Omega)$.
Hence $u - u_0 \in D(B^D)$ and $(B^D + \lambda \, I) (u - u_0) = - (\lambda - \lambda_0 - 1) \, u_0$.
Consequently 
\[
R_\lambda 
= \Big( I - (\lambda - \lambda_0 - 1) \, (B^D + \lambda \, I)^{-1} \Big) R_{\lambda_0 + 1}
\]
for all $\lambda \geq \lambda_0 + 1$.
Since the semigroup generated by $-B^D$ has Gaussian bounds, there exists a
$c_4 \geq 1$ such that 
 $\|(B^D + \lambda \, I)^{-1}\|_{\infty \to \infty} \leq c_4 \, \lambda^{-1}$ 
for all $\lambda \geq \lambda_0 + 1$.
Then $\|R_\lambda f\|_{L_\infty(\Omega)} \leq 2 c_3 \, c_4 \, \|f\|_{L_p(\Omega)}$
for all $\lambda \geq \lambda_0 + 1$ and $f \in L_p(\Omega)$.

Finally let $p' \in (2,4d)$ and let $q' \in (2,\infty)$ be such that 
$\frac{1}{q'} = \frac{1}{p'} - \frac{1}{4d}$.
There exists a $\theta \in (0,1)$ such that 
$\frac{1}{p'} = \frac{1-\theta}{2} + \frac{\theta}{p}$.
Then $\frac{1}{q'} = \frac{1-\theta}{q}$, where 
$\frac{1}{q} = \frac{1}{2} - \frac{1}{4d}$.
Let $c_1 > 0$ be as in Statement~\ref{lwreg207-1.5}.
The operator $R_\lambda$ is bounded from $L_2(\Omega)$ into 
$L_q(\Omega)$ with norm at most~$c_1$ by Statement~\ref{lwreg207-1.5},
and we just proved that the operator $R_\lambda$ is bounded from $L_p(\Omega)$ into 
$L_\infty(\Omega)$ with norm at most~$2 c_3 \, c_4$.
Hence by interpolation the operator $R_\lambda$ is bounded from 
$L_{p'}(\Omega)$ into $L_{q'}(\Omega)$ with norm bounded by 
$c_1^{1-\theta} \, (2 c_3 \, c_4)^\theta \leq c_1 + 2 c_3 \, c_4$,
which gives Statement~\ref{lwreg207-2}.

`\ref{lwreg207-3}'.
This is a special case of Corollary~\ref{cwreg206}.
\end{proof}

The main step in the proof of Proposition~\ref{pwreg209} is 
the next lemma.

\begin{lemma} \label{lwreg210}
There exist $\lambda > \lambda_0$ and $c > 0$ such that 
for all $\Phi \in C^1(\overline \Omega) \cap H^1(\Omega)$ there exists a unique
$u \in H^1(\Omega) \cap C(\overline \Omega)$ such that 
$u|_\Gamma = \Phi|_\Gamma$ and $\ca_\lambda u = 0$.
Moreover,
\[
\|u\|_{C(\overline \Omega)}
\leq c \, \|\Phi|_\Gamma\|_{C(\Gamma)}
.  \]
\end{lemma}
\begin{proof}
Choose $\delta = 0$ in Lemma~\ref{lwreg207}.
Let $c_1$ and $c_2$ be as in Lemma~\ref{lwreg207}.
Let $\lambda \in (\lambda_0 + 1,\infty)$ be such that 
$c_1 \, c_2^{2d-1} \, (\lambda - \lambda_0)^{-1/4} \, (1 + |\Omega|) \leq \frac{1}{2}$.
Let $R_\lambda$ be as in Lemma~\ref{lwreg207}.
Set $\varphi = \Phi|_\Gamma$.

There exist unique $w,\tilde w \in H^1_0(\Omega)$ such that 
$\gota_\lambda(w,v) = \gota_\lambda(\Phi,v)$ and
$\gotb_\lambda(\tilde w,v) = \gotb_\lambda(\Phi,v)$
for all $v \in H^1_0(\Omega)$.
Then $\tilde w \in C_0(\Omega)$ by Corollary~\ref{cwreg206}.
Define $u = \Phi - w$ and $\tilde u = \Phi - \tilde w$.
Then $\tilde u \in H^1(\Omega) \cap C(\overline \Omega)$ and 
$\tilde u|_\Gamma = \varphi$.
Moreover, $\gota_\lambda(u,v) = 0$ and $\gotb_\lambda(\tilde u,v) = 0$
for all $v \in H^1_0(\Omega)$,
and $\|\tilde u\|_{C(\overline \Omega)} \leq \|\varphi\|_{C(\Gamma)}$
by the result of Stampacchia mentioned in the introduction
(\cite{Stam2} Th\'eor\`eme~3.8).

Let $v \in H^1_0(\Omega)$.
Then 
\[
\gotb_\lambda(\tilde u - u,v)
= \sum_{k=1}^d (b_k \, u, \partial_k v)_{L_2(\Omega)}
   + (c_0 \, u, v)_{L_2(\Omega)}
\]
and 
$\tilde u - u = R_\lambda u$ by the definition of $R_\lambda$.

For all $n \in \{ 0,\ldots,2d \} $ define $p_n = \frac{4d}{2d-n}$.
Then $p_0 = 2$, $p_{2d-1} = 4d$, $p_{2d} = \infty$ 
and $\frac{1}{p_n} = \frac{1}{p_{n-1}} - \frac{1}{4d}$
for all $n \in \{ 1,\ldots,2d \} $.
So $\|\tilde u - u\|_{L_{p_n}(\Omega)} \leq c_2 \, \|u\|_{L_{p_{n-1}}(\Omega)}$
for all $n \in \{ 2,\ldots,2d \} $ and 
$\|\tilde u - u\|_{L_{p_1}(\Omega)} 
\leq c_1 \, (\lambda - \lambda_0)^{-1/4} \, \|u\|_{L_2(\Omega)}$
by Lemma~\ref{lwreg207}\ref{lwreg207-2} and \ref{lwreg207-1.5}.
Then 
\[
\|u\|_{L_{p_1}(\Omega)} 
\leq c_1 \, (\lambda - \lambda_0)^{-1/4} \, \|u\|_{L_2(\Omega)} 
   + (1 + |\Omega|) \, \|\tilde u\|_{L_\infty(\Omega)}
\]
and 
\[
\|u\|_{L_{p_n}(\Omega)} 
\leq c_2 \, \|u\|_{L_{p_{n-1}}(\Omega)} 
   + (1 + |\Omega|) \, \|\tilde u\|_{L_\infty(\Omega)}
\]
for all $n \in \{ 2,\ldots,2d \} $.
It follows by induction to $n$ that 
\[
\|u\|_{L_{p_n}(\Omega)} 
\leq c_1 \, c_2^{n-1} \, (\lambda - \lambda_0)^{-1/4} \, \|u\|_{L_2(\Omega)} 
   + (1 + |\Omega|) \sum_{k=0}^{n-1} c_2^k \, \|\tilde u\|_{L_\infty(\Omega)}
\]
for all $n \in \{ 2,\ldots,2d \} $.
So $u \in L_{p_{2d-1}}(\Omega) = L_{4d}(\Omega)$ and 
$\tilde u - u = R_\lambda u \in C_0(\Omega)$ by Lemma~\ref{lwreg207}\ref{lwreg207-3}.
In particular $u \in C(\overline \Omega)$.
Moreover, 
\begin{eqnarray*}
\|u\|_{L_\infty(\Omega)} 
& = & \|u\|_{L_{p_{2d}}(\Omega)}  \\
& \leq & c_1 \, c_2^{2d-1} \, (\lambda - \lambda_0)^{-1/4} \, \|u\|_{L_2(\Omega)} 
   + 2 d \, (1 + |\Omega|) \, c_2^{2d-1} \, \|\tilde u\|_{L_\infty(\Omega)}  \\
& \leq & c_1 \, c_2^{2d-1} \, (\lambda - \lambda_0)^{-1/4} \, (1 + |\Omega|) \, \|u\|_{L_\infty(\Omega)} 
   + 2 d \, (1 + |\Omega|) \, c_2^{2d-1} \, \|\tilde u\|_{L_\infty(\Omega)}  \\
& \leq & \frac{1}{2} \, \|u\|_{L_\infty(\Omega)} 
   + 2 d \, (1 + |\Omega|) \, c_2^{2d-1} \, \|\tilde u\|_{L_\infty(\Omega)}
\end{eqnarray*}
by the choice of $\lambda$.
So 
\[
\|u\|_{L_\infty(\Omega)} 
\leq 4 d \, (1 + |\Omega|) \, c_2^{2d-1} \, \|\tilde u\|_{L_\infty(\Omega)}
\leq 4 d \, (1 + |\Omega|) \, c_2^{2d-1} \, \|\varphi\|_{C(\Gamma)}
\]
and the proof of the lemma is complete.
\end{proof}

We next wish to remove the $\lambda$ in Lemma~\ref{lwreg210}.
For future purposes, we consider the full inhomogeneous problem.

\begin{prop} \label{pwreg214}
Let $p \in (d,\infty]$, $f_0 \in L_{p/2}(\Omega)$
and let $f_1,\ldots,f_d \in L_p(\Omega)$.
Let $u \in H^1_0(\Omega)$ be such that 
$\ca u = f_0 + \sum_{k=1}^d \partial_k f_k$.
Then $u \in C_0(\Omega)$.
\end{prop}
\begin{proof}
Without loss of generality we may assume that $p \in (d,4d)$.
Choose $\lambda = \delta = \lambda_0 + 1$ in Lemma~\ref{lwreg207} and in Proposition~\ref{pwreg203}.
By Proposition~\ref{pwreg203} there exists a unique $\tilde u \in H^1_0(\Omega) \cap C_0(\Omega)$
such that $\cb_\lambda \tilde u = f_0 + \sum_{k=1}^d \partial_k f_k$.
If $v \in C_c^\infty(\Omega)$, then 
\begin{eqnarray*}
\gotb_\lambda(\tilde u,v)
& = & \langle f_0 + \sum_{k=1}^d \partial_k f_k,v \rangle_{\cd'(\Omega) \times \cd(\Omega)}  \\
& = & \gota(u,v) \\
& = & \gotb_\lambda(u,v) 
   + \sum_{k=1}^d (b_k \, u, \partial_k v)_{L_2(\Omega)} 
   + ((c_0 - \delta \, \one_\Omega) \, u, v)_{L_2(\Omega)}
.  
\end{eqnarray*}
So 
\[
\gotb_\lambda(\tilde u - u,v) 
= \sum_{k=1}^d (b_k \, u, \partial_k v)_{L_2(\Omega)} 
   + ((c_0 - \delta \, \one_\Omega) \, u, v)_{L_2(\Omega)}
\]
and by density for all $v \in H^1_0(\Omega)$.
Hence $u - \tilde u = R_\lambda u$, where $R_\lambda$ is as in Lemma~\ref{lwreg207}.
For all $n \in \{ 0,\ldots,2d-1 \} $ define $p_n = \frac{4d}{2d-n}$.
Then $u - \tilde u \in L_2(\Omega) = L_{p_0}(\Omega)$.
It follows by induction to $n$ that $u \in L_{p_{n-1}}(\Omega)$ and 
$u - \tilde u \in L_{p_n}(\Omega)$ for all $n \in \{ 1,\ldots,2d-1 \} $, where
the last part follows from Lemma~\ref{lwreg207}\ref{lwreg207-2}.
Hence $u - \tilde u \in L_{p_{2d-1}}(\Omega) = L_{4d}(\Omega)$
and $u \in L_p(\Omega)$.
Then Lemma~\ref{lwreg207}\ref{lwreg207-3} gives  
$u - \tilde u = R_\lambda u \in C_0(\Omega)$ and therefore $u \in C_0(\Omega)$.
\end{proof}

\begin{cor} \label{cwreg208}
Let $p \in (d,\infty]$.
Then $(A^D)^{-1} (L_p(\Omega)) \subset C_0(\Omega)$.
\end{cor}

\begin{cor} \label{cwreg211} 
There exists a $c' > 0$ such that 
$\|(A^D)^{-1} f\|_{L_\infty(\Omega)} \leq c' \, \|f\|_{L_\infty(\Omega)}$
for all $f \in L_\infty(\Omega)$.
\end{cor}
\begin{proof}
Closed graph theorem.
\end{proof}

\begin{proof}[{\bf Proof of Proposition~\ref{pwreg209}.}]
Let $c,\lambda > 0$ be as in Lemma~\ref{lwreg210}
and let $c' > 0$ be as in Corollary~\ref{cwreg211}.
By Lemma~\ref{lwreg210} there exists a unique 
$\tilde u \in H^1(\Omega) \cap C(\overline \Omega)$
such that $\tilde u|_\Gamma = \Phi|_\Gamma$ and $\ca_\lambda \tilde u = 0$.
By Lemma~\ref{lwreg202} there exists a unique $w \in H^1_0(\Omega)$ such that 
$\gota(w,v) = \gota(\Phi|_\Omega,v)$ for all $v \in H^1_0(\Omega)$.
Set $u = \Phi|_\Omega - w$ and $\tilde w = \Phi|_\Omega - \tilde u$.
Then 
\begin{eqnarray*}
\gota(w,v)
& = & \gota(\Phi|_\Omega,v)
= \gota_\lambda(\Phi|_\Omega,v) - \lambda \, (\Phi,v)_{L_2(\Omega)}
= \gota_\lambda(\tilde w,v) - \lambda \, (\Phi,v)_{L_2(\Omega)}  \\
& = & \gota(\tilde w,v) 
   + \lambda \, (\tilde w,v)_{L_2(\Omega)}
   - \lambda \, (\Phi,v)_{L_2(\Omega)}
= \gota(\tilde w,v) 
   - \lambda \, (\tilde u,v)_{L_2(\Omega)}
\end{eqnarray*}
for all $v \in H^1_0(\Omega)$. 
So
\[
\gota(\tilde u - u, v) 
= \gota(w - \tilde w,v)
= - \lambda \, (\tilde u,v)_{L_2(\Omega)}
. \]
Since $\tilde u - u \in H^1_0(\Omega)$ it follows that
$A^D(\tilde u - u) = - \lambda \, \tilde u$.
Consequently, $u = \tilde u + \lambda \, (A^D)^{-1} \tilde u \in C_0(\Omega)$
by Corollary~\ref{cwreg208}.
Moreover,
\begin{eqnarray*}
\|u\|_{C(\overline \Omega)}
& = & \|u\|_{L_\infty(\Omega)}
\leq \|\tilde u\|_{L_\infty(\Omega)} + \lambda \, \|(A^D)^{-1} \tilde u\|_{L_\infty(\Omega)}  \\
& \leq & (1 + c' \, \lambda) \, \|\tilde u\|_{L_\infty(\Omega)}
\leq (1 + c' \, \lambda) \, c \, \|\Phi|_\Gamma\|_{C(\Gamma)}
\end{eqnarray*}
and the proof of Proposition~\ref{pwreg209} is complete.
\end{proof}

Define $|||\cdot||| \colon H^1_\loc(\Omega) \to [0,\infty]$ by
\[
|||u||| 
= \sup_{\delta > 0} 
  \sup_{\scriptstyle \Omega_0 \subset \Omega \; {\rm open} \atop
        \scriptstyle d(\Omega_0,\Gamma) = \delta}
    \delta \Big( \int_{\Omega_0} |\nabla u|^2 \Big)^{1/2}
.  \]
Finally we need the following Caccioppoli inequality.

\begin{prop} \label{pwreg213}
There exists a $c' \geq 1$ such that 
$|||u||| \leq c' \, \|u\|_{L_2(\Omega)}$
for all $u \in H^1(\Omega)$ such that $\ca u = 0$.
\end{prop}
\begin{proof}
See \cite{GiM} Theorem~4.4.
\end{proof}

Now we are able to prove Theorem~\ref{twreg101}.

\begin{proof}[{\bf Proof of Theorem~\ref{twreg101}.}]
The uniqueness is already proved in Proposition~\ref{pwreg202.5}.

Let $c > 0$ and $c' \geq 1$ be as in Propositions~\ref{pwreg209} and \ref{pwreg213}.
Let $\Phi \in C^1(\Ri^d) \cap H^1(\Ri^d)$.
By Proposition~\ref{pwreg209} there exists a unique 
$u \in H^1(\Omega) \cap C(\overline \Omega)$ such that 
$u|_\Gamma = \Phi|_\Gamma$ and $\ca u = 0$.
Moreover,
\begin{eqnarray}
\|u\|_{C(\overline \Omega)} + |||u|||
& \leq & \|u\|_{C(\overline \Omega)} + c' \, \|u\|_{L_2(\Omega)}  \nonumber  \\
& \leq & (2 + |\Omega|) \, c' \, \|u\|_{C(\overline \Omega)}   \nonumber  \\
& \leq & (2 + |\Omega|) \, c \, c' \, \|\Phi|_\Gamma\|_{C(\Gamma)}
.
\label{etwreg101;4}
\end{eqnarray}
It follows from (\ref{etwreg101;4}) that we can define a linear map
$F \colon \{ \Phi|_\Gamma : \Phi \in C^1(\Ri^d) \cap H^1(\Ri^d) \} \to H^1(\Omega) \cap C(\overline \Omega)$
by $F(\Phi|_\Gamma) = u$, where $u \in H^1(\Omega) \cap C(\overline \Omega)$ 
is such that 
$u|_\Gamma = \Phi|_\Gamma$ and $\ca u = 0$.
Now let $\varphi \in C(\Gamma)$.
By the Stone--Weierstra\ss\ theorem there are 
$\Phi_1,\Phi_2,\ldots \in C^1(\Ri^d) \cap H^1(\Ri^d)$ such that 
$\lim \Phi_n|_\Gamma = \varphi$ in $C(\Gamma)$.
Set $u_n = F(\Phi_n|_\Gamma)$ for all $n \in \Ni$.
Then it follows from (\ref{etwreg101;4}) that $(u_n)_{n \in \Ni}$
is a Cauchy sequence in $C(\overline \Omega)$.
Let $u = \lim u_n$ in $C(\overline \Omega)$.
Also $(u_n)_{n \in \Ni}$ is a Cauchy sequence in $H^1_\loc(\Omega)$ by 
(\ref{etwreg101;4}).
So $u \in H^1_\loc(\Omega)$.
Since $\ca u_n = 0$  for all $n \in \Ni$, one deduces that 
$\ca u = 0$.
Moreover, $u|_\Gamma = \lim u_n|_\Gamma = \lim \Phi_n|_\Gamma = \varphi$.
This proves existence.
Finally, 
\[
\|u\|_{C(\overline \Omega)}
= \lim \|u_n\|_{C(\overline \Omega)}
\leq \lim (2 + |\Omega|) \, c \, c' \, \|\Phi_n|_\Gamma\|_{C(\Gamma)}
= (2 + |\Omega|) \, c \, c' \, \|\varphi\|_{C(\Gamma)}
.  \]
This completes the proof of Theorem~\ref{twreg101}.
\end{proof}

Theorem~\ref{twreg101} has the following extension.

\begin{thm} \label{twreg216}
Adopt the notation and assumptions of Theorem~\ref{twreg101}.
Let $\varphi \in C(\Gamma)$, $p \in (d,\infty]$, $f_0 \in L_{p/2}(\Omega)$
and let $f_1,\ldots,f_d \in L_p(\Omega)$.
Then there exists a unique $u \in C(\overline \Omega) \cap H^1_\loc(\Omega)$
such that $u|_\Gamma = \varphi$ and $\ca u = f_0 + \sum_{k=1}^d \partial_k f_k$.
\end{thm}
\begin{proof}
The uniqueness follows as in the proof of Proposition~\ref{pwreg202.5}.

By Lemma~\ref{lwreg202} there exists a $u_0 \in H^1_0(\Omega)$ such that 
$\ca u_0 = f_0 + \sum_{k=1}^d \partial_k f_k$.
Then $u_0 \in C_0(\Omega)$ by Proposition~\ref{pwreg214}.
By Theorem~\ref{twreg101} there exists a $u_1 \in C(\overline \Omega) \cap H^1_\loc(\Omega)$
such that $u_1|_\Gamma = \varphi$ and $\ca u_1 = 0$.
Define $u = u_0 + u_1$.
Then $u \in C(\overline \Omega) \cap H^1_\loc(\Omega)$.
Moreover, $u|_\Gamma = \varphi$ and $\ca u = f_0 + \sum_{k=1}^d \partial_k f_k$.
\end{proof}

We conclude this section with some results for the classical solution.
They will be used in Section~\ref{Swreg3new} and are of independent interest.
Recall that $P \colon C(\Gamma) \to C(\overline \Omega)$ is given by
$P \varphi = u$, where $u \in C(\overline \Omega) \cap H^1_\loc(\Omega)$ is 
the classical solution, so $u|_\Gamma = \varphi$ and $\ca u = 0$.

\begin{prop} \label{pwreg240}
Let $\Phi \in C(\overline \Omega) \cap H^1_\loc(\Omega)$.
Suppose there exists a $w \in H^1_0(\Omega)$ such that $\ca \Phi = \ca w$.
Then $w \in C(\overline \Omega)$ and $P(\Phi|_\Gamma) = \Phi - w$.
\end{prop}
\begin{proof}
Write $\tilde w = \Phi - P(\Phi|_\Gamma)$.
Then $\tilde w \in C_0(\Omega) \cap H^1_\loc(\Omega)$
and $\ca \tilde w = \ca \Phi = \ca w = f_0 + \sum_{k=1}^d \partial_k f_k$,
where $f_0 = c_0 \, w + \sum_{l=1}^d c_l \, \partial_l w \in L_2(\Omega)$
and $f_k = - \sum_{l=1}^d a_{lk} \, \partial_l w - b_k \, w \in L_2(\Omega)$
for all $k \in \{ 1,\ldots,d \} $.
So $\tilde w \in H^1_0(\Omega)$ by Lemma~\ref{lwreg205}.
Hence $\ca(\tilde w - w) = $ and $\tilde w - w \in \ker A^D = \{ 0 \} $.
So $w = \tilde w = \Phi - P(\Phi|_\Gamma)$.
\end{proof}

We need the dual map of $\ca$.
Define the map $\ca^t \colon H^1_\loc(\Omega) \to \cd'(\Omega)$ by
\[
\langle \ca u,v \rangle_{\cd'(\Omega) \times \cd(\Omega)}
= \sum_{k,l=1}^d \int_\Omega a_{lk} \, (\partial_k u) \, \overline{\partial_l v}
   - \sum_{k=1}^d \int_\Omega \overline{c_k} \, u \, \overline{\partial_k v}
   - \sum_{k=1}^d \int_\Omega \overline{b_k} \, (\partial_k u) \, \overline v
   + \int_\Omega \overline{c_0} \, u \, \overline v
\]
for all $u \in H^1_\loc(\Omega)$ and $v \in C_c^\infty(\Omega)$.

\begin{cor} \label{cwreg241}
Suppose that $a_{kl}, b_k, c_k \in W^{1,\infty}(\Omega)$ for all 
$k,l \in \{ 1,\ldots,d \} $.
Let $\Phi \in C(\overline \Omega)$.
Suppose there exists a $w \in H^1_0(\Omega)$ such that 
\[
\langle \Phi, \ca^t v\rangle_{\cd'(\Omega) \times \cd(\Omega)}
= \gota(w,v)
\]
for all $v \in C_c^\infty(\Omega)$.
Then $w \in C(\overline \Omega)$ and $P(\Phi|_\Gamma) = \Phi - w$.
\end{cor}
\begin{proof}
By assumption one has $\langle \Phi - w, \ca^t v\rangle_{\cd'(\Omega) \times \cd(\Omega)} = 0$
for all $v \in C_c^\infty(\Omega)$.
Hence $\Phi - w \in H^1_\loc(\Omega)$ by elliptic regularity.
So $\Phi \in H^1_\loc(\Omega)$ and 
\[
\langle \ca \Phi, v\rangle_{\cd'(\Omega) \times \cd(\Omega)}
= \langle \Phi, \ca^t v\rangle_{\cd'(\Omega) \times \cd(\Omega)}
= \gota(w,v)
= \langle \ca w, v\rangle_{\cd'(\Omega) \times \cd(\Omega)}
\]
for all $v \in C_c^\infty(\Omega)$.
Therefore $\ca \Phi = \ca w$ and the result follows from 
Proposition~\ref{pwreg240}.
\end{proof}

The last corollary takes a very simple form for the Laplacian.

\begin{cor} \label{cwreg242}
Let $\Phi \in C(\overline \Omega)$.
Suppose that $\Delta \Phi \in H^{-1}(\Omega)$.
Let $w \in H^1_0(\Omega)$ be such that $\Delta \Phi = \Delta w$ as distribution.
Then $w \in C(\overline \Omega)$ and $P(\Phi|_\Gamma) = \Phi - w$.
\end{cor}

This corollary is a special case of \cite{AD2} Theorem~1.1.

\section{Variational and classical solutions: comparison}  \label{Swreg3new}

In this section we show that the variational and classical solutions
of the Dirichlet problem are the same.
For that we assume throughout this section that $\Omega$ is an open 
set with Lipschitz boundary.
Moreover, we adopt the assumptions and notation of Theorem~\ref{twreg101}.
Recall that for all $\varphi \in C(\Gamma)$ we denote by $P \varphi \in C(\overline \Omega)$
the classical solution and for all $\varphi \in H^{1/2}(\Gamma)$, 
we denote by $\gamma \varphi \in H^1(\Omega)$ the variational solution
of the Dirichlet problem.
We shall prove in this section that they coincide if both are defined.

The fact that they coincide for restrictions to $\Gamma$ of functions in 
$C(\overline \Omega) \cap H^1(\Omega)$ is a consequence of Proposition~\ref{pwreg240}.
We state this as a proposition.

\begin{prop} \label{pwreg411}
Let $\Phi \in C(\overline \Omega) \cap H^1(\Omega)$.
Then $P (\Phi|_\Gamma) = \gamma (\Phi|_\Gamma)$ almost everywhere.
\end{prop}

So for the proof of Theorem~\ref{twreg102} it suffices to show that 
the map $\Phi \mapsto \Phi|_\Gamma$ from $C(\overline \Omega) \cap H^1(\Omega)$
into $C(\Gamma) \cap H^{1/2}(\Gamma)$ is surjective.
This is surprisingly difficult to prove.
We first prove Theorem~\ref{twreg102} for the Laplacian with the help of 
Proposition~\ref{pwreg411} and a deep result of Dahlberg.
As a consequence we obtain the desired surjectivity result.
Then as noticed earlier, Theorem~\ref{twreg102} follows for our general
elliptic operator.

\begin{thm} \label{tdtnc401}
Assume that $a_{kl} = \delta_{kl}$ and $b_k = c_k = c_0 = 0$
for all $k,l \in \{ 1,\ldots,d \} $.
Let $\varphi \in C(\Gamma) \cap H^{1/2}(\Gamma)$.
Then $P \varphi = \gamma \varphi$ almost everywhere.
\end{thm}
\begin{proof}
Let $x \in \Omega$.
By Dahlberg \cite{Dahlberg} Theorem~1 there exists a unique $k_x \in L_1(\Gamma)$ such that 
$(P \varphi)(x) = \int_\Gamma k_x \, \varphi \, d\sigma$
for all $\varphi \in C(\Gamma)$.

Now let  $\varphi \in C(\Gamma) \cap H^{1/2}(\Gamma)$.
Without loss of generality we may assume that $\varphi$ is real valued.
Then there exists a $u \in H^1(\Omega,\Ri)$ such that $\varphi = \Tr u$.
Since $H^1(\Omega) \cap C(\overline \Omega)$ is dense in $H^1(\Omega)$, there 
exist $u_1,u_2,\ldots \in H^1(\Omega,\Ri) \cap C(\overline \Omega)$ such that 
$\lim u_n = u$ in $H^1(\Omega)$.
Define $v_n = (-\|\varphi\|_{L_\infty(\Gamma)}) \vee u_n \wedge \|\varphi\|_{L_\infty(\Gamma)}$
for all $n \in \Ni$.
Then $v_n \in H^1(\Omega) \cap C(\overline \Omega)$.
Write $\varphi_n = v_n|_\Gamma = \Tr v_n \in C(\Gamma) \cap H^{1/2}(\Gamma)$
for all $n \in \Ni$.
Then $P \varphi_n = \gamma \varphi_n$ almost everywhere for all $n \in \Ni$
by Proposition~\ref{pwreg411}.

Note that 
\[
\lim \varphi_n 
= \lim \Tr v_n 
= (-\|\varphi\|_{L_\infty(\Gamma)}) \vee \Tr u \wedge \|\varphi\|_{L_\infty(\Gamma)}
= \varphi
\]
in $H^{1/2}(\Gamma)$.
So by continuity of $\gamma$ one deduces that 
$\gamma \varphi = \lim \gamma \varphi_n$ in $H^1(\Omega)$
and in particular in $L_2(\Omega)$.
Passing to a subsequence, if necessary, we may assume that 
\[
(\gamma \varphi)(x) = \lim (\gamma \varphi_n)(x)
\]
for almost all $x \in \Omega$.
Using again that $\lim \varphi_n = \varphi$ in $H^{1/2}(\Gamma)$
and therefore also in $L_2(\Gamma)$,
we may assume that 
$\lim \varphi_n = \varphi$ almost everywhere on $\Gamma$.
Hence if $x \in \Omega$, then 
\[
(P \varphi)(x)
= \int_\Gamma k_x \, \varphi \, d\sigma
= \lim \int_\Gamma k_x \, \varphi_n \, d\sigma
= \lim (P \varphi_n)(x)
\]
by the Lebesgue dominated convergence theorem.
Since $P \varphi_n = \gamma \varphi_n$ almost everywhere for all $n \in \Ni$
one concludes that $(P \varphi)(x) = (\gamma \varphi)(x)$ for almost all $x \in \Omega$.
\end{proof}

The desired surjectivity result is the following corollary
of Theorem~\ref{tdtnc401}.

\begin{cor} \label{cdtnc402}
Let $\Omega \subset \Ri^d$ be a bounded open set with Lipschitz boundary.
Let $\varphi \in C(\Gamma) \cap H^{1/2}(\Gamma)$.
Then there exists a $u \in H^1(\Omega) \cap C(\overline \Omega)$
such that $\varphi = u|_\Gamma$.
\end{cor}

\begin{proof}[{\bf Proof of Theorem~\ref{twreg102}.}]
This follows from Corollary~\ref{cdtnc402} and Proposition~\ref{pwreg411}.
\end{proof}

\begin{cor} \label{cwreg305}
Adopt the notation and assumptions of Theorem~\ref{twreg101}.
Suppose that $\Omega$ has a Lipschitz boundary.
Let $u \in C(\overline \Omega) \cap H^1_\loc(\Omega)$ and 
suppose that $\ca u = 0$.
Then $u \in H^1(\Omega)$ if and only if $u|_\Gamma \in H^{1/2}(\Gamma)$.
\end{cor}
\begin{proof}
`$\Rightarrow$' is trivial.

`$\Leftarrow$'.
Suppose $u|_\Gamma \in H^{1/2}(\Gamma)$.
Then $u = P(u|_\Gamma) = \gamma (u|_\Gamma) \in H^1(\Omega)$ by Theorem~\ref{twreg102}.
\end{proof}

\section{Semigroup and holomorphy on $C_0(\Omega)$} \label{Swreg4new}

In this section we prove Theorem~\ref{twreg103}.
Throughout this section we adopt the notation and assumptions of Theorem~\ref{twreg101}.
We need several lemmas.

\begin{lemma} \label{lwreg310}
The operator $A_c$ is invertible and $(A_c)^{-1} = (A^D)^{-1}|_{C_0(\Omega)}$.
\end{lemma}
\begin{proof}
If $v \in C_0(\Omega)$, then $(A^D)^{-1} v \in C_0(\Omega)$ by 
Corollary~\ref{cwreg208}.
Moreover, $A^D ((A^D)^{-1} v) = v$.
So $(A^D)^{-1} v \in D(A_c)$ and $A_c ((A^D)^{-1} v) = v$.
Hence $A_c$ is surjective.
Since $A^D$ is injective, also $A_c$ is injective.
Therefore $A_c$ is invertible and 
$(A_c)^{-1} = (A^D)^{-1}|_{C_0(\Omega)}$.
\end{proof}

The next proof is inspired by arguments in \cite{ABenilan} Theorem~4.4.

\begin{lemma} \label{lwreg311}
The domain $D(A_c)$ of the operator $A_c$ is dense in $C_0(\Omega)$.
\end{lemma}
\begin{proof}
Let $\rho \in M(\Omega)$, the Banach space of all 
complex measures on $\Omega$ and suppose that $\int_\Omega v \, d\rho = 0$
for all $v \in D(A_c)$.
There exist $w_1,w_2,\ldots \in L_2(\Omega)$ such that 
$\sup \|w_n\|_{L_1(\Omega)} < \infty$ and 
$\lim \int_\Omega v \, \overline{w_n} = \int_\Omega v \, d\rho$
for all $v \in C_0(\Omega)$.

Choose $p = d+2$ and let $q \in (1,2)$ be the dual exponent of $p$.
It follows from Proposition~\ref{pwreg214} that the operator $(A^D)^{-1}$ 
extends to a continuous operator from $W^{-1,p}(\Omega)$ into $C_0(\Omega)$.
Hence the operator $(A^D)^{-1*}$ 
extends to a continuous operator from $M(\Omega)$ into $W^{1,q}_0(\Omega)$.
In particular, there exists a $c > 0$ such that 
$\|(A^D)^{-1*} w\|_{W^{1,q}_0(\Omega)} \leq c \, \|w\|_{L_1(\Omega)}$
for all $w \in L_2(\Omega)$.
For all $n \in \Ni$ set $u_n = (A^D)^{-1*} w_n$.
We emphasise that $u_n \in D((A^D)^*)$.
Then $\sup \|u_n\|_{W^{1,q}_0(\Omega)} < \infty$.
Note that $W^{1,q}_0(\Omega)$ is reflexive.
Hence passing to a subsequence if necessary, 
there exists a $u \in W^{1,q}_0(\Omega)$
such that $\lim u_n = u$ weakly in $W^{1,q}_0(\Omega)$.

Let $v \in C_c^\infty(\Omega)$.
Then $(A^D)^{-1} v \in D(A_c)$ by Lemma~\ref{lwreg310}.
Therefore 
\begin{eqnarray*}
0 
= \int_\Omega (A^D)^{-1} v \, d\rho
& = & \lim \int_\Omega \Big( (A^D)^{-1} v \Big) \, \overline{w_n}  \\
& = & \lim (v , (A^D)^{-1*} w_n)_{L_2(\Omega)}
= \lim (v , u_n)_{L_2(\Omega)}
= \lim \int_\Omega v \, \overline{u_n}
= \lim \int_\Omega v \, \overline u
. 
\end{eqnarray*}
Hence $u = 0$.

Again let $v \in C_c^\infty(\Omega)$.
Then 
\[
\int_\Omega v \, d\rho
= \lim \int_\Omega v \, \overline{w_n}  
= \lim (v, (A^D)^* u_n)_{L_2(\Omega)}  
= \lim \gota(v, u_n)
= 0
,  \]
where we used (\ref{eSwreg1;4}).
So $\rho = 0$ and $D(A_c)$ is dense in $C_0(\Omega)$.
\end{proof}

Now we prove that $-A_c$ generates a holomorphic $C_0$-semigroup.

\begin{proof}[{\bf Proof of Theorem~\ref{twreg103}.}]
Let $S$ be the semigroup generated by $-A^D$.
Then $S$ has a kernel with Gaussian upper bounds by \cite{Ouh5} Theorem~6.10
(see also \cite{Daners} Theorem~6.1 for operators with 
real valued coefficients and \cite{AE1} Theorems~3.1 and 4.4).
Hence the semigroup $S$ extends consistently 
to a semigroup $S^{(p)}$ on $L_p(\Omega)$
for all $p \in [1,\infty]$.

Choose $p \in (d,\infty]$.
Let $t > 0$ and $u \in L_2(\Omega)$.
Since $S$ is a holomorphic semigroup, one deduces that $S_t u \in D(A^D)$
and $A^D \, S_t u \in L_2(\Omega)$.
Next the Gaussian kernel bounds imply that $S_t$ maps $L_2(\Omega)$ 
into $L_p(\Omega)$. 
So $A^D \, S_{2t} u = S_t \, A^D \, S_t u \in L_p(\Omega)$
and 
\begin{equation}
S_{2t} u \in (A^D)^{-1} (L_p(\Omega)) \subset C_0(\Omega)
\label{etwreg103;10}
\end{equation}
by Corollary~\ref{cwreg208}.
Hence $S_t C_0(\Omega) \subset C_0(\Omega)$ for all $t > 0$.
For all $t > 0$ let $S^c_t = S_t|_{C_0(\Omega)} \colon C_0(\Omega) \to C_0(\Omega)$.
Then $(S^c_t)_{t > 0}$ is a semigroup on $C_0(\Omega)$.
Moreover, using again the Gaussian kernel bounds
there exists an $M \geq 1$ such that 
$\|S^c_t\| \leq \|S^{(\infty)}_t\| \leq M$
for all $t \in (0,1]$.

Let $t \in (0,1]$ and $u \in D(A_c)$.
Then 
\[
\|(I - S^c_t) u\|_{C_0(\Omega)}
= \| \int_0^t S^s \, A_c u \, ds\|_{C_0(\Omega)}
\leq \int_0^t M \, \|A_c u\|_\infty \, ds
= M \, t \, \|A_c u\|_\infty
.  \]
So $\lim_{t \downarrow 0} S^c_t u = u$ in $C_0(\Omega)$.
Since $D(A_c)$ is dense in $C_0(\Omega)$ by Lemma~\ref{lwreg311}, 
one deduces that $\lim_{t \downarrow 0} S^c_t u = u$ in $C_0(\Omega)$
for all $u \in C_0(\Omega)$.
So $S^c$ is a $C_0$-semigroup.

Finally, using once more the Gaussian kernel bounds,
it follows that the semigroup $S^c$ is holomorphic (see \cite{AE1} Theorem~5.4).
\end{proof}

We conclude this section by establishing Gaussian kernels which are continuous 
up to the boundary.
For this we use the following special case of \cite{AE8} Theorem~2.1.

\begin{prop} \label{pwreg350}
Suppose that $|\partial \Omega| = 0$.
Let $T$ be a semigroup in $L_2(\Omega)$ such that 
$T_t L_2(\Omega) \subset C(\overline \Omega)$ and 
$T_t^* L_2(\Omega) \subset C(\overline \Omega)$
for all $t > 0$.
Then for all $t > 0$ there exists a unique $k_t \in C(\overline \Omega \times \overline \Omega)$
such that 
\[
(T_t u)(x) = \int_\Omega k_t(x,y) \, u(y) \, dy
\]
for all $u \in L_2(\Omega)$ and $x \in \Omega$.
\end{prop}

We continue to denote by $S$ the semigroup generated by $-A^D$ 
and we also denote by $S$ the holomorphic extension.
For all $\theta \in (0,\pi]$ let 
$\Sigma(\theta) = \{ z \in \Ci \setminus \{ 0 \} : |\arg z| < \theta \} $
be the open sector with (half)angle~$\theta$.

\begin{thm} \label{twreg351}
Adopt the notation and assumptions of Theorem~\ref{twreg101}.
In addition assume that $b_k$ is real valued for all $k \in \{ 1,\ldots,d \} $.
Let $\theta$ be the holomorphy angle of $S$.
Then for all $z \in \Sigma(\theta)$ there exists a unique
$k_z \in C(\overline \Omega \times \overline \Omega)$ such that the following is 
valid.
\begin{tabelR}
\item \label{twreg351-1}
$(S_z u)(x) = \int_\Omega k_z(x,y) \, u(y) \, dy$ 
for all $z \in \Sigma(\theta)$, $u \in L_2(\Omega)$ and $x \in \overline \Omega$.
\item \label{twreg351-2}
$k_z(x,y) = 0$ for all $z \in \Sigma(\theta)$ and 
$x,y \in \overline \Omega$ with $x \in \partial \Omega$ or $y \in \partial \Omega$.
\item \label{twreg351-3}
The map $z \mapsto k_z$ is holomorphic from $\Sigma(\theta)$ into
$C(\overline \Omega \times \overline \Omega)$.
\item \label{twreg351-4}
For all $\theta' \in (0,\theta)$ there exist $b,c,\omega > 0$ such that 
\[
|k_z(x,y)|
\leq c \, |z|^{-d/2} \, e^{\omega |z|} \, e^{- b \frac{|x-y|^2}{|z|}}
\]
for all $z \in \Sigma(\theta')$ and $x,y \in \overline \Omega$.
\end{tabelR}
\end{thm}
\begin{proof}
It follows from (\ref{etwreg103;10}) that $S_z L_2(\Omega) \subset C_0(\Omega)$
for all $z \in \Sigma(\theta)$.
Since the coefficients $b_k$ are real, also the adjoint operator
satisfies the conditions of Theorem~\ref{twreg101}.
Therefore $S_z^* L_2(\Omega) \subset C_0(\Omega)$
for all $z \in \Sigma(\theta)$.
It follows from Proposition~\ref{pwreg350} that for all 
$z \in \Sigma(\theta)$ there exists a unique
$k_z \in C(\overline \Omega \times \overline \Omega)$ such that
$(S_z u)(x) = \int_\Omega k_z(x,y) \, u(y) \, dy$ 
for all $u \in L_2(\Omega)$ and $x \in \overline \Omega$.
Since $S_z u \in C_0(\Omega)$ one deduces that $k_z(x,y) = 0$
for all $z \in \Sigma(\theta)$, $x \in \partial \Omega$ and 
$y \in \overline \Omega$.
Considering adjoints the same is valid with $x$ and $y$ interchanged.
If $v,w \in C_0(\Omega)$, then the map
\[
z \mapsto 
\langle k_z, v \otimes \overline w 
       \rangle_{C(\overline \Omega \times \overline \Omega) \times C(\overline \Omega \times \overline \Omega)^*}
= (S_z u, v)_{L_2(\Omega)}
\]
is holomorphic on $\Sigma(\theta)$.
Therefore Statement~\ref{twreg351-3} is a consequence of \cite{ArN} Theorem~3.1.
The Gaussian bounds of Statement~\ref{twreg351-4} follow from 
\cite{AE1} Theorem~5.4.
\end{proof}

\subsection*{Acknowledgements}
The second-named author is most grateful for the hospitality extended
to him during a fruitful stay at the University of Ulm.
He wishes to thank the University of Ulm for financial support. 
Part of this work is supported by an
NZ-EU IRSES counterpart fund and the Marsden Fund Council from Government funding,
administered by the Royal Society of New Zealand.
Part of this work is supported by the
EU Marie Curie IRSES program, project `AOS', No.~318910.

\end{document}